\documentclass[amstex,12pt]{article}

\usepackage{amsmath}
\usepackage{amsthm}
\usepackage{amsbsy}

\newtheorem{theo}{Theorem}
\newtheorem{cor}{Corollary}[theo]

\begin{document}
\renewcommand{\thefootnote}{\fnsymbol{footnote}}
\begin{center}
\Large \bf The Last Digit of $\boldsymbol{\ \binom{2n}{n}\ \text{\Large and}\ 
 \displaystyle \pmb{\sum}\textstyle \binom{n}{i}\binom{2n-2i}{n-i}}\vspace{.3in}$

\normalfont \large  Walter Shur

20 Speyside Circle

Pittsboro, NC 27312

wrshur@gmail.com
%wshur @@ worldnet.att.net \vspace{24 pt}

%Submitted: June 28, 1996; Accepted: November 11, 1996. 

%AMS subject classification (1991): Primary 05A10; Secondary 11B65.\vspace{.2in}

\end{center}

\abstract{Let \mbox{$f_{n}= \displaystyle \sum_{i=0}^n \textstyle \binom{n}{i}\binom{2n-2i}{n-i}$}, \mbox{$g_{n}= \displaystyle \sum_{i=1}^n \textstyle \binom{n}{i}\binom{2n-2i}{n-i}$}. Let $\{a_k\}_{k=1}$ be the set of all positive integers n, in increasing order, for which $\binom{2n}{n}$ is not divisible by 5, and let $\{b_k\}_{k=1}$ be the set of all positive integers n, in increasing order, for which $g_n$ is not divisible by 5. This note finds simple formulas for $a_k$, $b_k$, $ \binom{2n}{n}\ mod\ 10$, $ f_{n}\ mod\ 10$, and $ g_{n}\ mod\ 10$.}

\vspace{.2in}

\bf Definitions \vspace{.1in}\\

 \normalfont \mbox{$f_{n}= \displaystyle \sum_{i=0}^n \textstyle \binom{n}{i}\binom{2n-2i}{n-i}$};\quad \mbox{$g_{n}= \displaystyle \sum_{i=1}^n \textstyle \binom{n}{i}\binom{2n-2i}{n-i}$}\\
  
$\{a_k\}_{k=1}$ is the set of all positive integers n, in increasing order, for which $\binom{2n}{n}$ is not divisible by 5. \\

$\{b_k\}_{k=1}$ is the set of all positive integers n, in increasing order, for which $g_n$ is not divisible by 5.\\

$u_n$ is the number of unit digits in the base 5 representation of~n~. 
\pagebreak

\begin{theo}

$a_{k}$ is the number in base 5 whose digits represent the number $k$ in base 3. If $n \ge 1$, 
\[\binom{2n}{n}\ mod\ 10=\left\{\begin{array}{l}
     \left.\begin{array}{l}0 \end{array}\right.\ \ \ if\ n \not \in \{a_k\}\\
     \left.\begin{array}{l}
            2\\4\\6\\8 \end{array}\right\}\ if\ n  \in \{a_k\}\ \text{and}\ u_n\ mod\ 4=\left\{\begin{array}{l}
1\\2\\0\\3\end{array}\right.\\
\end{array}\right..\]
Note that if $n \in \{a_k\},\ u_n$ is odd (even) if and only if n is odd (even). 
\end{theo}	

\begin{proof}	

\par 	From Lucas' theorem [1], we have \[\binom{2n}{n}\equiv 
	\binom{N_1}{n_1}\binom{N_2}{n_2} \cdots 	\binom{N_t}{n_t}\mod{5},\] 
	where  $2n=(N_r \cdots 	N_3 N_2 	N_1)_5$, $n=(n_s 	\cdots n_3 n_2 n_1)_5$, and $t=min(r,s)$.

\par	Suppose that for each $i\leq t$,\ $n_i\leq 2$. Then, 	for each $i\leq t$,\ $N_i=2n_i$. Since $n_i=0,1$ or 	$2$, each term of the product$ 	\binom{N_1}{n_1}\binom{N_2}{n_2} \cdots 	\binom{N_t}{n_t}$ is $1,2$ or $6$. Hence, $\binom{2n}n$ 	is not divisible by 5.

\par	Suppose that for some i, $n_i>2$. Let $i_m$ be the 	smallest value of i for which that is true. Then, if 	$n_{i_m}$ is 3 or 4, $N_{i_m}$ is 1 or 3 (resp.). In 	either case, $\binom{N_{i_m}}{n_{i_m}}=0$, and 	$\binom{2n}{n}$ is divisible by 5. 	
	
\par	Thus,  $\{a_k\}$ is the set of all positive integers written in base 3, but interpreted as if they were written in base 5. Since $\{a_k\}$\ is in increasing order, the first part of the theorem is proved.  

Suppose now that $\binom{2n}{n}$ is not divisible by 5. Then each term of the product$ \binom{N_1}{n_1}\binom{N_2}{n_2} \cdots \binom{N_t}{n_t}$ is $1,2$ or\ $6$\ (according as $n_i=0,1,\ or\ 2$). We have, noting that $\binom{2n}{n}$ is even,
\stepcounter{footnote}
\vspace{.13in}\[\left .\begin{array}{c} 2^{u_n}\ mod\ 10=6\footnotemark,\ 2,\ 4\ \text{or}\ 8, \\
               \phantom{x} \\
 \binom{N_1}{n_1}\binom{N_2}{n_2} \cdots \binom{N_t}{n_t}\ mod\ 10=6,\ 2,\ 4\ \text{or}\ 8,\\
\phantom{x}\\
 \binom{2n}{n}\ mod\ 10=6,\ 2,\ 4\ \text{or}\ 8,
 \end{array}\right \}\text{according as}\ u_n\ mod\ 4=0,\ 1,\ 2\ \text{or}\ 3.  \]\footnotetext{Equals 1 if $u_n=0$; nevertheless, the next line follows since, if $u_n=0,$ at least one $n_i$ must equal 2, making $\binom{2n_i}{n_i}=6.$ }

\end{proof}	
\vspace{.1in}	
\begin{cor}

\[a_k=k+2\displaystyle \sum_{i=1}  \textstyle \left \lfloor\frac{k}{3^i}\right \rfloor 5^{i-1}.\]
	
\end{cor}
\pagebreak
\begin{proof}

Let $k=(\cdots d_3d_2d_1)_3$, and consider $a_k=\displaystyle \sum_{i=1}
\textstyle d_i 5^{i-1}$.
	
$\begin{array}{ccccc}
	d_1 &=&k &-&3\left\lfloor\frac{k}{3}\right\rfloor\vspace{5pt}\\
d_2&=&\left\lfloor\frac{k}{3}\right\rfloor&-&3\left\lfloor\frac{k}{3^2}\right\rfloor\vspace{5pt} \\
    d_3&=&\left\lfloor\frac{k}{3^2}\right\rfloor&-&3\left\lfloor\frac{k}{3^3}
\right\rfloor 	\\
	\vdots &\ &\vdots &\  & \vdots
\end{array}$

\vspace{.3in}Therefore, 
\[	\sum_{i=1}d_i 5^{i-1}=\sum_{i=1}\left(\left\lfloor\frac{k}{3^{i-1}}\right\rfloor-
	3\left\lfloor\frac{k}{3^i}\right\rfloor \right )5^{i-1}.
\]

\vspace{.3in}Since $\left\lfloor\frac{k}{3^i}\right\rfloor5^i-3\left\lfloor\frac{k}{3^i}\right
\rfloor5^{i-1}=2  \left \lfloor\frac{k}{3^i}\right\rfloor 5^{i-1}$, the corollary is proved.

\end{proof}

\begin{cor}
Let $\mu_k$ be the largest integer $t$ such that $k/3^t$ is an integer. Then,
\[a_k-a_{k-1}=\frac{5^{\mu_k}+1}{2},\  and\    a_k=1+\sum_{i=2}^k\frac{5^{\mu_i}+1}{2}.\]
 
$\mu_k=m$\  if and only if $k \in \{j3^m\},$\ where $j$ is a positive integer and\ $j\ mod\ 3\ne 0.$	

\end{cor}

\begin{proof}

If $\mu_k>0$, then \[k=(\cdots d_{\mu_k +1}0\cdots 0)_3;\quad d_{\mu_k+1}\ge 1;\quad \text{and}\  k-1=(\cdots (d_{\mu_k+1}-1)2\cdots 2)_3.\]
Hence, \[a_k-a_{k-1}=5^{\mu_k}-2[5^{\mu_k-1}+5^{\mu_k-2}+\cdots+1]=\frac{5^{\mu_k}+1}
{2}.\]
If $\mu_k=0$,\ then \[k=(\cdots d_1)_3;\quad d_1\ge 1;\quad\text{and} \ k-1=(\cdots (d_1-1))_3.\]
Hence,\[a_k-a_{k-1}=1=\frac{5^{\mu_k}+1}{2}.\]
\nopagebreak
\vspace{.1in}The remaining parts of the corollary follow immediately. 

\end{proof}

\begin{cor}

If $k>1$,
\begin{equation*}
a_k=
  \begin{cases}
   5a_{\frac{k}{3}}& \text{if $k\ mod\ 3=0$},\\
	a_{k-1}+1& \text{if $k\ mod\ 3 \ne 0$}.
  \end{cases}
\end{equation*}
\end{cor}

\begin{proof}

If $k\ mod\ 3=0,\ \text{then}\ \ k=(\cdots d_20)_3\ \ \text{and} \ \ \frac{k}{3}=(\cdots
 d_2)_3.\ \text{Hence,}\ a_k~=~5a_{\frac{k}{3}}.$

If\ $k\ mod\ 3\ne 0,\ \text{then}\ \mu_k=0\ \text{and from Corollary 1.2, we have}\ \ a_k~-~a_{k-1}~=~1.$  

\end{proof}

\vspace{.2in}\begin{theo}
	$b_k$ is the number in base 5 whose digits represent the number $2k-1$ in base 3, i.e. $b_k=a_{2k-1}$. Furthermore,\ $g_n\ mod\ 10$ can only take on the values 1,5 or 9, as follows:
\[g_n\ mod\ 10=\left\{\begin{array}{l}
     \left.\begin{array}{l}5 \end{array}\right.\ \ \ if\ n \not \in \{b_k\}\\
     \left.\begin{array}{l}
            1\\9 \end{array}\right\}\ if\ n  \in \{b_k\}\ \text{and}\ u_n\ mod\ 4=\left\{\begin{array}{l}
1\\3\end{array}\right.
\end{array}\right..\]

\end{theo}

\begin{proof}

Let $F(z)=\displaystyle \sum_n f_n z^n=\displaystyle \sum_n z^n\sum_i \binom{n}{i}\binom{2n-2i}{n-i}.$

\vspace{.15in}Letting t=n-i, we have

\begin{align*}
F(z)&=\sum_nz^n\sum_t\binom{n}{t}\binom{2t}{t}\\
&=\sum_t\binom{2t}{t}\sum_n\binom{n}{t}z^n\\
&=\frac{1}{1-z}\sum_t\binom{2t}{t}\left (\frac{z}{1-z}\right )^t\qquad \text{(see [2])}\\
&=\frac{1}{1-z}\frac{1}{\sqrt{1-\frac{4z}{1-z}}}=\frac{1}{\sqrt{1-z}}\frac{1}{\sqrt{1-5z}}\qquad \text{(see [2])}\\
&=[1+(\frac{1}{4})\binom{2}{1}z+(\frac{1}{4})^2\binom{4}{2}z^2+\cdots]
[1+(\frac{1}{4})\binom{2}{1}5z+(\frac{1}{4})^2\binom{4}{2}5^2z^2+\cdots].
\end{align*}

Hence, \[f_n=\frac{1}{4^n}\sum_{i=0}\binom{2i}{i}
\binom{2n-2i}{n-i}5^i,\]
and
\begin{align*}
g_n&=\frac{1}{4^n}\sum_{i=0}\binom{2i}{i}\binom{2n-2i}{n-i}5^i-\binom{2n}{n},\\
&\phantom{=}\\
&=\frac{\displaystyle \sum_{i=1}\binom{2i}{i}\binom{2n-2i}{n-i}5^i-(4^n-1)\binom{2n}{n}}
{4^n}.
\end{align*}

Thus we see that\ $g_n$ is divisible by 5 if and only if\  $(4^n-1)\binom{2n}{n}$\ is divisible by 5. And since $g_n$ is odd, $g_n$ mod $10=5$ if and only if $g_n$ is divisible by 5.\ $4^n-1$\  is divisible by 5 if and only if n is even. Therefore, $g_n$ mod $10\ne 5$ if and only if n is odd and $n \in \{a_k\}$.\ Hence,\ $b_k=a_{2k-1},$\ from which it follows that\ $b_k$\ is the number in base 5 whose digits represent the number 2k-1 in base 3. 

Suppose that $g_n\ mod\ 10\ne 5.\ \text{Then }\ \binom{2n}{n}\ mod\ 10=c,$\ where (since $n \in \{a_k\}$ and n and $u_n$ are odd) c is 2 or 8, according as\ $u_n$\ mod $4=1\ \text{or}\ 3.$  Thus, for some non-negative integers j and k,\ $4^n-1=10j+3$\ and\ $\binom{2n}{n}=10k+c. $\ Since\ $\binom{2i}{i}$\ is even when\ $i\ge 1$,\ for some non-negative integer q we have
\[ 4^n g_n=10q-(10j+3)(10k+c).\] 

Since\ $g_n$\ is odd, and $4^n\ mod\ 10=4,$\ we have

\vspace{.13in}If c=2,\ $g_n\ mod\ 10=1;$\\
\indent if c=8,\ $g_n\ mod\ 10=9.$   

\end{proof}   

\begin{cor}

\[b_k=2k-1+2\displaystyle \sum_{i=1} \textstyle \left\lfloor\frac{2k-1}{3^i}\right\rfloor 5^{i-1}.\]

\end{cor}

\begin{proof}

This follows from Corollary 1.1, since $b_k=a_{2k-1}$.

\end{proof}

\pagebreak

\begin{cor}

Let\ $\nu_k$ be the largest integer $t$ for which $\frac{(k-1)(2k-1)}{3^t}$ is an integer. Then,
\[b_k-b_{k-1}=\frac{5^{\nu_k}+3}{2},\  and\    b_k=1+\sum_{i=2}^k\frac{5^{\nu_i}+3}{2}.\]
\vspace{3pt} 
If $m \ge 1$, $\nu_k=m$\  if and only if $k \in \left\{\left\lceil \frac{j3^m+1}{2}\right\rceil \right\},$\ where $j$ is a positive \vspace{3pt}\\integer and\ $j\ mod\ 3\ne 0$; if $m=0$,\ $\nu_k=m$\  if and only if $k \in \{3j\}$,\vspace{3pt}\\ where $j$ is a positive integer.  	

\end{cor}

\begin{proof}

\begin{gather*}
b_k=a_{2k-1},\\
b_k-b_{k-1}=(a_{2k-1}-a_{2k-2})+(a_{2k-2}-a_{2k-3}),\\
b_k-b_{k-1}=\frac{5^{\mu_{2k-1}}+1}{2}+\frac{5^{\mu_{2k-2}}+1}{2},
\end{gather*}
where $\mu_k$ is the largest integer $t$ such that $k/3^t$ is an integer.
\vspace{3pt} 

Note that $\nu_k$ is also the largest integer t for which\   $\frac{(2k-1)(2k-2)}{3^t}$\ is an integer.
Then we must have one of the following cases:
\begin{align*}
\mu_{2k-1}&=0\ \text{and}\ \mu_{2k-2}=0,\ \text{or}\\
\mu_{2k-1}&=\nu_k\ \text{and}\ \mu_{2k-2}=0,\ \text{or}\\
\mu_{2k-1}&=0\ \text{and}\ \mu_{2k-2}=\nu_k.
\end{align*}

In any of these cases,\ \[b_k-b_{k-1}=\frac{5^{\nu_k}+3}{2}.\]

If $m\ge 1$, at most one of\ $(k-1)$\ and\ $(2k-1)$ is divisible by $3^m$.\  $\nu_k=m$\ if and only if either $(k-1)$ or $(2k-1)$ is divisible by $3^m$ but not by $3^{m+1}$. Suppose\ $m\ge 1$\ and\ $j$\ mod\ $3\ne 0$\vspace{.1in}.

If \,  j \,is odd, $\left\lceil \frac{j3^m+1}{2}\right\rceil =\frac{j3^m+1}{2};$  
\vspace{.1in}\quad if $k=\frac{j3^m+1}{2},\ 2k-1=j3^m,\ \text{and}\ \nu_k=~m.$

\vspace{.1in}If\,  j\,  is even, $\left\lceil \frac{j3^m+1}{2}\right\rceil =\frac{j3^m+2}{2};$  
\vspace{.1in}\quad if $k=\frac{j3^m+2}{2},\ k-1=\frac{j3^m}{2},\ \text{and}\ \nu_k=~m.$

It is straightforward to show the converse, that if \mbox{$\nu_k=m \ge1,\ k\in~ \left\{\left\lceil \frac{j3^m+1}{2}\right\rceil \right\}$.}\vspace{.1in}

If\ $m=0$,\ $\nu_k=m$ if and only if neither\ $(k-1)$\ or\ $(2k-1)$\ is a multiple of 3. This occurs when\ $2k$\ (and therefore k) is a multiple of 3. 

\end{proof}

\begin{cor}

If $k\ge 1$,
\begin{alignat*}{2}
     b_{3k}\phantom{+}&=b_{3k-1}+2,& \qquad \\
   b_{3k+1}&=5b_{k+1}-4,&\qquad k\ mod\ 3&=0,\\
   	\phantom{b_{3k+1}}&=b_{3k}+4,&\qquad  k\ mod\ 3 &\ne 0,\\
   b_{3k+2}&=5b_{k+1}.& \qquad 
  \end{alignat*}

\end{cor}

\begin{proof}
\begin{align*}
b_{3k}&=a_{6k-1}=a_{6k-2}+1=a_{6k-3}+2=b_{3k-1}+2,\\
b_{3k+2}&=a_{6k+3}=5a_{2k+1}=5b_{k+1},\\
b_{3k+1}&=a_{6k+1}=a_{6k}+1=5a_{2k}+1,\ \text{and}
\end{align*}

if k mod $3=0$,
\[b_{3k+1}=5(a_{2k+1}-1)+1=5b_{k+1}-4;\]

if k mod $3\ne 0$,
\[b_{3k+1}=5(a_{2k-1}+1)+1=5b_k+6=b_{3k-1}+6=(b_{3k}-2)+6=b_{3k}+4.\]
\end{proof}

\begin{theo}

\[f_n\ mod\ 10=\left\{\begin{array}{l}
     \left.\begin{array}{l}5 \end{array}\right.\ \ \ if\ n \not \in \{a_k\}\\
     \left.\begin{array}{l}
            1\\3\\7\\9 \end{array}\right\}\ if\ n  \in \{a_k\}\ \text{and}\ u_n\ mod\ 4=\left\{\begin{array}{l}
0\\1\\3\\2\end{array}\right.
\end{array}\right..\]

\end{theo}
\vspace{.1in}\begin{proof}
Since $f_n=\binom{2n}{n}+g_n$, the corollary can be proved easily by combining the results of Theorem 1 and Theorem 2.
\end{proof}
\vspace{.6in}
\pagebreak
\begin{center}
\bf \Large References\normalfont\vspace{.2in}
\end{center}

\flushleft [1] I. Vardi, Computational Recreations in Mathematica, Addison-Welsey, California, 1991, p.70 (4.4).

\flushleft [2] H.S. Wilf, generatingfunctionology (1st ed.), Academic Press,\\ New York, 1990, p.50 (2.5.7, 2.5.11).

\vspace{2in}

\end{document}